\documentclass[12pt]{article}
\usepackage{amsthm,amsmath,amssymb,
extsizes}

\usepackage[thinlines,thiklines]{easymat}
\renewcommand{\ge}{\geqslant}

\newcommand{\cir}{
\begin{picture}(6,6)
\put(3,3){\circle*{1.5}}
\end{picture}}

\newcommand{\righteqn}{
\begin{picture}(0,0)
 \put(-9,-1){$\left\{\rule{0pt}{25pt}\right.$}
 \end{picture}}

\newcommand{\lin}{ \,\frac{}{\quad\ }\,}

\newtheorem{theorem}{Theorem}
\newtheorem{lemma}[theorem]{Lemma}

\theoremstyle{definition}
\newtheorem{definition}[theorem]{Definition}

\theoremstyle{remark}
\newtheorem{remark}[theorem]{Remark}

\newtheorem{example}[theorem]{Example}

\title{Pairs of mutually annihilating
operators\footnotetext{This is the authors' version of a work that was published in \emph{Linear Algebra Appl.} 430 (2009) 86--105}}

\author{
Vitalij M. Bondarenko
\\ Institute of Mathematics,
Tereshchenkivska 3, Kiev,
Ukraine\\vit-bond@imath.kiev.ua
  \and
Tatiana  G. Gerasimova\\
Mech.-Math. Faculty,
Kiev National University,\\
Vladimirskaja 64, Kiev, Ukraine\\
smyarga@mail.ru
 \and
Vladimir V. Sergeichuk%
\thanks{Corresponding author.}\\
Institute of Mathematics,
Tereshchenkivska 3, Kiev,
Ukraine\\sergeich@imath.kiev.ua}
\date{}

\begin{document}
\maketitle

\begin{abstract}
Pairs $(\cal A,B)$ of mutually annihilating operators ${\cal A}{\cal B}={\cal B}{\cal A}=0$ on a finite dimensional vector space over an algebraically closed field were classified by Gelfand and Ponomarev [\emph{Russian Math. Surveys} 23 (1968) 1--58] by method of linear relations.
The classification of $(\cal A,B)$ over any field was derived by
Nazarova, Roiter, Sergeichuk, and Bondarenko [\emph{J. Soviet Math.} 3 (1975) 636--654] from the classification of finitely generated modules over a dyad of two local Dedekind rings. We give canonical matrices of $(\cal A,B)$ over any field in an explicit form and our proof is constructive: the matrices of $(\cal A,B)$ are sequentially reduced to their canonical form by similarity transformations $(A,B)\mapsto (S^{-1}AS, S^{-1}BS)$.


\end{abstract}

\section{Introduction}

We consider the problem of classifying
pairs of mutually annihilating operators
\[
{\cal A, B}: V\to V,\qquad
{\cal A}{\cal B}={\cal B}{\cal A}=0
\]
on a finite dimensional vector space $V$.

The pairs $({\cal A, B})$
were classified
\begin{itemize}
  \item
in \cite{gel}
over an algebraically closed field
by method of linear relations, and
  \item
in \cite{n-r-s-b,lau} over any field $\mathbb F$  as modules over $\mathbb F[x,y]/(xy)$;
\end{itemize}
these results are surveyed in Remark \ref{rrr}.

Our classification of $({\cal A, B})$ over any field is constructive: we give an algorithm for reducing its matrices to canonical form by similarity transformations
\begin{equation}\label{fek1}
(A,B)\mapsto S^{-1}(A,B)S:= (S^{-1}AS, S^{-1}BS),\qquad S\text{ is nonsingular}.
\end{equation}

Our paper was inspired by
Oblak's article \cite{Obl}, in which she characterizes all possible pairs of Jordan canonical forms $(J_A,J_B)$ for pairs $(A,B)$ of mutually annihilating matrices $AB=BA=0$ over an algebraically closed field. For this purpose, she puts one matrix in Jordan form, then she uses only those similarity transformations that preserve it and reduces the second matrix to a simple form. We continue to reduce the second matrix until obtain a canonical form of $(A,B)$.

In Section \ref{s1} we formulate the only theorem of this paper:
we classify pairs $(\cal A,B)$  of mutually annihilating operators and give a canonical form of their matrix pairs $(A,B)$.
In Sections \ref{sub1}--\ref{sub3} we prove this theorem and reduce $(A,B)$ to its canonical form (see the end of Section \ref{s1}).

\begin{remark} \label{rrr}

Pairs $(\cal A,B)$ of mutually annihilating operators were classified by different methods:

(i) Gelfand and Ponomarev \cite[Chapter 2]{gel} classified
the pairs $(\cal A,B)$ over  an algebraically closed field by using the apparatus of MacLane's theory of linear relations. They arrived to this problem studding indecomposable representations of SL$(2,\mathbb C)$. Using their classification of $(\cal A,B)$,
Schr\"{o}er \cite{sch} classified the irreducible components of the varieties $V(n,a,b)$ of pairs $(A,B)$ of $n\times n$ matrices satisfying
$
AB=BA=A^a=B^b=0.
$

(ii) Nazarova and Roiter \cite{n-r} classified finitely generated modules over a dyad $D$ of two local Dedekind rings. In the subsequent paper \cite{n-r-s-b}, Bondarenko, Nazarova, Roiter, and Sergeichuk   corrected two inaccuracies in \cite{n-r} and derived classifications
\begin{itemize}
  \item
of
finite $p$-groups
possessing an abelian subgroup of index $p$, by
taking $D=\mathbb Z_p[x]/(x^p)$ where $\mathbb Z_p$ is the ring of $p$-adic numbers, and

  \item
of pairs $(\cal A,B)$ of mutually annihilating operators over \emph{any} field $\mathbb F$ by
taking $D=\mathbb F[x,y]/(xy)$.
\end{itemize}
It is a very  curious circumstance that two classification problems,  so  unlike  at  first glance, admit of a  like  solution.  Levy \cite{lev,lev1} extended the classification of modules over $D$ to modules over Dedekind-like rings.

(iii) Laubenbacher and Sturmfels \cite{lau} also derived a classification of
$(\cal A,B)$ from a classification of finitely generated modules over $D=\mathbb F[x,y]/(xy)$. They used the presentation of each finitely generated module $M$ over $D=\mathbb F[x,y]/(xy)$ as a quotient of a free module
\[
D^n\stackrel{f}{\longrightarrow} D^m\longrightarrow M \longrightarrow 0,
\]
where $f$ corresponds to a matrix $A(x, y)$ with entries in $D$. This presentation is nonunique: $f$ can be multiplied on the left by an automorphism of $D^m$ and on the right by an automorphism of $D^n$. Each automorphism of $D^m$ is given by a nonsingular matrix over $D$; thus, a polynomial matrix $A(x, y)$ can be reduced by elementary transformations over $D$. (By an analogous method, the problem of classifying finitely generated modules over \emph{any} finite dimensional algebra can be reduced to a matrix problem, see \cite{dro2} and \cite[Section 2.5]{ser_can}.) Laubenbacher and Sturmfels \cite{lau} developed an \emph{algorithm} that transforms
$A(x, y)$ to a normal form, which is analogous to the Smith normal form for a matrix over a polynomial ring in one
variable. Their algorithm partially uses
the matrix reduction carried out in \cite{n-r-s-b}. The ring $D=\mathbb F[x,y]/(xy)$ also appears in a variety of other contexts,
such as $K$-theory \cite{den,hes} and algebraic geometry \cite[Lemma 4.5]{kl}.
\end{remark}

\begin{remark}
The classification of pairs of mutually annihilating operators is a bit surprise because
\begin{itemize}
  \item
the problem of classifying \emph{arbitrary} pairs of operators $(\cal A,B)$ is considered as hopeless since it contains the problem of classifying \emph{any} system of linear operators (i.e., representations of an arbitrary quiver); see, for example, \cite{bel-ser_compl,gel-pon}, and

  \item
the commutativity condition ${\cal A}{\cal B}={\cal B}{\cal A}$ does not simplify the problem of classifying $(\cal A,B)$ since by \cite{gel-pon} the classification of pairs of commuting operators implies the classification of pairs of arbitrary operators. Indeed, two pairs $(A,B)$ and $(C,D)$ of $n\times n$ matrices are similar if and only if two pairs of commuting and nilpotent matrices
\[\arraycolsep=1mm
\left(\begin{bmatrix}
0&0&0&0\\0&0&0&0\\I&0&0&0\\0&I&A&0
\end{bmatrix}, \begin{bmatrix}
0&0&0&0\\I&0&0&0\\0&0&0&0\\0&B&I&0
\end{bmatrix}\right),\quad
\left(\begin{bmatrix}
0&0&0&0\\0&0&0&0\\I&0&0&0\\0&I&C&0
\end{bmatrix}, \begin{bmatrix}
0&0&0&0\\I&0&0&0\\0&0&0&0\\0&D&I&0
\end{bmatrix}\right)
\]
are similar
(all blocks are $n\times n$).
\end{itemize}
Nevertheless, Belitskii's algorithm \cite{bel1,ser_can} converts an arbitrary pair $(A,B)$ of $n\times n$ matrices to some pair $(A_{\text{can}},B_{\text{can}})$ by similarity transformations such that two pairs $(A,B)$ and $(A',B')$ are similar if and only if \[(A_{\text{can}},B_{\text{can}})= (A'_{\text{can}},B'_{\text{can}}).\] Thus, the pair $(A_{\text{can}},B_{\text{can}})$ can be considered as a canonical form of $(A,B)$ for similarity, but there is no satisfactory description of the set of matrix pairs $(A_{\text{can}},B_{\text{can}})$.
The algorithm presented in Sections  \ref{sub1}--\ref{sub3} is a special case of Belitskii's algorithm.
\end{remark}

\section{Canonical form of matrices of a pair of mutually annihilating operators}
\label{s1}

All vector spaces and matrices that we consider are over an arbitrary field $\mathbb F$.

Let us define two types of pairs of mutually
annihilating operators
\begin{equation*}\label{gro}
 {\cal A}:V\to V,\quad {\cal
B}:V\Rightarrow V,\qquad
 {\cal A} {\cal B} ={\cal B} {\cal A}=0,
\end{equation*}
on a vector space $V$ (to distinguish the operators, we use a
\emph{double arrow} $\Rightarrow$ for $\cal B$).

\begin{definition}\label{kie}
A pair of mutually annihilating
operators  ${\cal A}:V\to V$ and ${\cal
B}:V\Rightarrow V$ is of \emph{path type} if it is defined as follows. Let
\begin{equation}\label{jut}
1 \lin
2\lin 3\lin \cdots\lin (t-1)
\lin t \qquad (t\ge 1)
\end{equation}
be any path graph
in which every edge is an
ordinary arrow $\longrightarrow$  or a
double arrow $\Longleftarrow$ (with this
orientation). Take
\[V:=\mathbb Fe_1\oplus \mathbb
Fe_2\oplus\dots\oplus\mathbb Fe_t\]
and define the action of $\cal A$ and $\cal B$ on the basis vectors
$e_1,\dots,e_t$
by \eqref{jut}, in which every vertex $i$ is replaced by $e_i$ and the unspecified action is zero. The matrix pair
\begin{equation}\label{yte}
 (A,B)
\end{equation}
that gives $\cal A$ and $\cal B$ in the basis $e_1,\dots,e_t$ is called a \emph{matrix pair of path type}.
\end{definition}

Clearly the pair \eqref{yte} is formed by mutually annihilating $t\times t$ matrices $A=[a_{ij}]$ and $B=[b_{ij}]$, in which
\[
\left.\begin{matrix}
a_{i+1,i}=1&\text{if $i\longrightarrow (i+1)$}\\
b_{i,i+1}=1&\text{if $i\Longleftarrow (i+1)$}
\end{matrix}\right\}\text{ in } \eqref{jut},\ \ i=1,\dots,t-1,
\]
and the other entries are zero. Note that
\[
A+B^T=\begin{bmatrix}
0&&&0\\1&0&&\\&\ddots&\ddots&\\0&&1&0\\
\end{bmatrix}.
\]

\begin{example}
The path  graph
\[
1\longrightarrow 2 \longrightarrow 3 \Longleftarrow 4
\]
defines the following action of $\cal A$ and $\cal B$ on the basis vectors:
\[
e_1\stackrel{\cal A}{\longrightarrow} e_2 \stackrel{\cal A}{\longrightarrow} e_3 \stackrel{\;\cal B}{\Longleftarrow} e_4
\]
and the pair $(\cal A,B)$ is
given by the matrix pair
\[
(A,B)=\left(\begin{bmatrix}
0&0&0&0\\1&0&0&0\\0&1&0&0\\0&0&0&0\\
\end{bmatrix},\, \begin{bmatrix}
0&0&0&0\\0&0&0&0\\0&0&0&1\\0&0&0&0\\
\end{bmatrix}\right).
\]
\end{example}

Recall that every square matrix $A$ is similar
to a direct sum of {\it
Frobenius blocks}
\begin{equation}\label{3}
\begin{bmatrix} 0&&
0&-c_n\\1&\ddots&&\vdots
\\&\ddots&0&-c_2\\
0&&1& -c_1 \end{bmatrix},
\end{equation}
in which $p(x)^l=x^n+c_1
x^{n-1}+\dots+ c_n$ is an
integer power of a polynomial
$p(x)$ that is irreducible over
$\mathbb F$ (note that $p(x)^l$ is the minimal polynomial of \eqref{3}). This direct sum is uniquely determined by $A$, up to permutation of summands; see \cite[Section 14]{pra}. If $\mathbb F$ is algebraically closed then $p(x)=x-\lambda $ and the reader may use the $n$-by-$n$ Jordan block
\begin{equation*}\label{jye}
J_n(\lambda):=\begin{bmatrix}
\lambda  &&&0\\1&\lambda &&
\\&\ddots&\ddots&
\\ 0&&1&\lambda
\end{bmatrix}
\end{equation*}
instead of \eqref{3} in all the statements of this paper.

\begin{definition}
 \label{rsi}
 A pair of mutually annihilating
operators ${\cal A}:V\to V$ and ${\cal
B}:V\Rightarrow V$ is of \emph{cycle type} if it is defined as follows.
\begin{itemize}
  \item[(i)]
Let
\begin{equation}\label{jut1}
\begin{picture}(0,0)(54,0)
\put(0,0){$1\lin 2\lin \cdots\lin t$}
\qbezier(10,-1.67)(51,-20)(92,-1.67)
\end{picture}
\end{equation}
\\
be a cycle graph in which every straight
edge is $\longrightarrow$ or $\Longleftarrow$
and the arcuated edge is $\longleftarrow$ or
$\Longrightarrow$ (with this orientation).

  \item[(ii)]
Let this graph be \emph{aperiodic}, this means that the cyclic renumbering of its vertices
\\[-9mm]

\begin{equation}\label{kiyi}
\begin{picture}(0,0)(150,0)
\put(0,0){$1\lin 2\lin \cdots\lin t \quad \to\quad i\lin (i+1)\lin \cdots\lin (i-1) $}
\qbezier(10,-1.67)(51,-20)(93,-1.67)
\qbezier(150,-1.67)(205,-20)(261,-1.67)
\end{picture}
\end{equation}
\\
is not an isomorphism for each $i=2,\dots, t$. In other words, for each nontrivial rotation of this cyclic graph there is an ordinary or double arrow that is mapped to a double or, respectively, ordinary arrow.

  \item[(iii)]
By (ii), if \eqref{jut1} has no double arrow, then it is the loop $\circlearrowleft$; we associate with its arrow a
nonsingular Frobenius block
$\Phi$ (or a nonsingular Jordan block if \/$\mathbb
F$ is algebraically closed). If the graph has a
double arrow, then we choose any double arrow and associate $\Phi$ with it.
\end{itemize}
Let  $k\times k$ be the size of $\Phi$. Define the action of $\cal A$ and $\cal B$ on the $kt$-dimensional vector space
\begin{equation*}\label{gtr}
V:=V_1\oplus\dots\oplus V_t,\qquad
V_i:=\mathbb
Fe_{i1}\oplus\dots\oplus\mathbb
Fe_{ik},
\end{equation*}
by \eqref{jut1}, in which each
vertex $i$ is replaced by $V_i$
and each arrow represents the linear mapping of the corresponding vector spaces. This linear mapping is given by $\Phi$ if the arrow has been associated with $\Phi$; otherwise, it is given by
the identity matrix $I_k$.

The matrix pair
$ (A,B)$
that gives $\cal A$ and $\cal B$ in the basis
\[
e_{11},\, \dots,\,e_{1k};\,\dots;\,e_{t1},\, \dots,\,e_{tk}
\]
is called a \emph{matrix pair of cycle type}.
Thus,
\begin{itemize}
  \item
if $t=1$ and the loop $1\lin 1$ is an ordinary arrow (which is associated with $\Phi$), then $(A,B)=(\Phi,0_k)$;

  \item
if $t=1$ and the loop $1\lin 1$ is a double arrow  (which is associated with  $\Phi$), then $(A,B)=(0_k,\Phi)$;

  \item
if $t\ge 2$, then $A=[A_{ij}]$ and $B=[B_{ij}]$ are
block matrices (consisting of $t^2$ blocks and each block is of size  $k\times k$), in which for $i=1,\dots,t$:
\begin{equation}\label{gre}
 A_{i+1,i}=
      I_k\quad \text{if $i\longrightarrow (i+1)$},
  \qquad
B_{i,i+1}=
  \begin{cases}
    I_k & \text{if $i\Longleftarrow (i+1)$} \\
    \Phi  & \text{if $i\stackrel{\;\Phi }{\Longleftarrow} (i+1)$}
  \end{cases}
\end{equation}
(if $i=t$ then all $i+1$ in \eqref{gre} are replaced by $1$); the other blocks of $A$ and $B$ are zero. Note that
\[
A+B^T=\begin{bmatrix}
0_k&\dots&\dots&0_k&*\\ *&0_k&&&0_k\\
0_k&*&0_k&&\vdots\\
\vdots&\ddots&\ddots&\ddots&\vdots\\
0_k&\dots&0_k&*&0_k
\end{bmatrix}\qquad (t^2 \text{ blocks})
\]
in which $0_k$ is the $k\times k$ zero matrix, one star is $\Phi $ and the others are $I_k$.
\end{itemize}

\end{definition}

\begin{example}
The cycle graph\\[-7mm]

\begin{equation}\label{huo}
\begin{picture}(0,0)(55,0)
\put(0,0){$1 \Longleftarrow 2\stackrel{\;\Phi}{\Longleftarrow} 3 \longrightarrow 4$}
\put(10,-2){\vector(-3,1){.07}}
\qbezier(10,-2)(50.5,-20)(91,-2)
\end{picture}
\end{equation}
\\
($\Phi$ is 3-by-3) defines the following action of $\cal A$ and $\cal B$ on the basis $e_{i1},\dots, e_{ik}$ of each space $V_i$:\\[-7mm]

\[
\begin{picture}(0,0)(62,0)
\put(0,0){$V_1 \stackrel{\;I_3}{\Longleftarrow} V_2\stackrel{\;\Phi}{\Longleftarrow} V_3 \stackrel{\;I_3}{\longrightarrow} V_4$}
\put(16,-2){\vector(-3,1){.07}}
\qbezier(16,-2)(62,-20)(108,-2)
\put(62,-20){$\scriptstyle I_3$}
\end{picture}
\]
\\[4mm]
and the pair $(\cal A,B)$ is
given by the matrix pair
\[(A,B)=\left(
\begin{bmatrix}
0_3&0_3&0_3&I_3\\0_3&0_3&0_3&0_3 \\0_3&0_3&0_3&0_3
\\0_3&0_3&I_3&0_3\\
\end{bmatrix},\, \begin{bmatrix}
0_3&I_3&0_3&0_3\\0_3&0_3&\Phi&0_3
\\0_3&0_3&0_3&0_3
\\0_3&0_3&0_3&0_3\\
\end{bmatrix}\right).
\]
\end{example}

Let ${\cal P}:=({\cal A}, {\cal B})$ and ${\cal P}':=({\cal A}', {\cal B}')$ be two pairs of linear operators on vector spaces $V$ and $V'$, respectively. Define their \emph{direct sum}
\begin{equation*}\label{rwe}
({\cal A},{\cal B})\oplus({\cal A}',{\cal B}')
:= ({\cal A}\oplus{\cal A}', {\cal B}\oplus{\cal B}')\quad \text{on } V\oplus V'.
\end{equation*}
We say that ${\cal P}$ is \emph{isomorphic} to ${\cal P}'$ if there exists a linear bijection $\varphi: V\to V'$ transforming ${\cal P}$ to ${\cal P}'$; that is,
\[
\varphi{\cal A}= {\cal A}'\varphi,\qquad \varphi{\cal B}= {\cal B}'\varphi.
\]

\begin{theorem}\label{theor}
{\rm(a)}
Let $\cal A$ and $\cal B$ be two
 linear operators on a vector space
 over any field $\mathbb F$, and let
\begin{equation}\label{ffg}
{\cal A}{\cal B}={\cal B}{\cal A}=0.
\end{equation}
Then $({\cal A},{\cal B})$ is isomorphic to a direct sum of
pairs of path and cycle types and
this sum is uniquely determined by $({\cal A},{\cal B})$, up to
\begin{itemize}
  \item[\rm(i)]
permutation of direct summands and
  \item[\rm(ii)]
replacing any summand given by a cycle graph
\eqref{jut1}
with the pair given by any
cycle graph obtained from  \eqref{jut1}
\begin{itemize}
  \item
by a cyclic renumbering of its vertices \eqref{kiyi}
and/or
  \item
if there are at least two double arrows then by transferring
$\Phi $ $($associated with one double arrow$)$ to another double arrow.
\end{itemize}
\end{itemize}

{\rm(b)}
Each pair $(A,B)$ of mutually annihilating matrices
\begin{equation}\label{ffg1}
AB=BA=0
\end{equation}
 is similar to  a direct sum of matrix
pairs of path and cycle types and
this sum is uniquely determined by $(A,B)$, up to transformations {\rm(i)} and {\rm(ii)}.
\end{theorem}

For example, the cycle graph \eqref{huo} and the cycle graph
\\[-7mm]

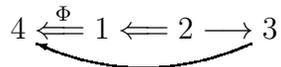
\[
\begin{picture}(0,0)(50,0)
\put(0,0){$4 \stackrel{\;\Phi}{\Longleftarrow} 1\; {\Longleftarrow}\; 2 \longrightarrow 3$}
\put(10,-2){\vector(-3,1){.07}}
\qbezier(10,-2)(50,-20)(91,-2)
\end{picture}
\]
\\
give isomorphic pairs of cyclic type.

\begin{remark}
\begin{itemize}
  \item[\rm(a)]
The pair of path type given by \eqref{jut} can be also given briefly by
the sequence
\[(c_1,\dots,c_{t-1})
\]
in which
\begin{equation}\label{ytj}
c_i:=
  \begin{cases}
    1 & \text{if the $i$th
arrow is ordinary}, \\
    2 & \text{if the $i$th
arrow is double}.
  \end{cases}
\end{equation}

  \item[\rm(b)]
The pair of cycle type given by \eqref{jut1} can be also given,  up to change of basis, by the system
\[
(c_1,\dots,c_{t}; \Phi )
\]
in which the sequence
$(c_1,\dots,c_{t})$ (defined by \eqref{ytj})
is aperiodic and is determined up to cyclic permutation.

\end{itemize}
\end{remark}

In the remaining sections we construct an algorithm that converts a pair $(A,B)$ of mutually annihilating matrices to its canonical form defined in Theorem \ref{theor}(b).
\begin{itemize}
  \item
In Section \ref{sub1} we reduce the general case to the case of nilpotent $A$, convert $A$ to its Jordan canonical form, restrict ourselves to those similarity transformations that preserve $A$, and show that they induce on some submatrix $D$ of $B$ (containing all nonzero entries of $B$) a matrix problem solved in \cite{n-r,n-r-s-b}.

  \item
In Section \ref{sub2} we apply the reduction described in \cite{n-r,n-r-s-b} and transform $D$ to a block form such that each horizontal or vertical strip contains at most one nonzero block, and this block is nonsingular.
  \item
In Section \ref{sub3}, extending the partition of $D$ into blocks, we find a block form of $A$ and $B$ such that each horizontal or vertical strip contains at most one nonzero block, and this block is nonsingular. This implies the decomposition of the corresponding  operator pair $({\cal A},{\cal B})$ into a direct sum of
pairs of path and cycle types, which proves Theorem \ref{theor}.

\end{itemize}

\section{Reduction to a chessboard matrix problem}\label{sub1}

Let us start to reduce a pair $(A,B)$ of mutually annihilating matrices by similarity transformations \eqref{fek1} to its canonical form described in Theorem \ref{theor}(b).

\begin{lemma}\label{lem9}
{\rm(a)}
Each pair of mutually annihilating matrices $(A,B)$ is similar to a direct sum
\begin{equation}\label{teu}
(A',B')\oplus(\Phi_1,0_{n_1})\oplus\dots\oplus (\Phi_r,0_{n_r}),
\end{equation}
in which $A'$ is nilpotent and each $\Phi_i$ is an $n_i\times n_i$ nonsingular Frobenius block.

{\rm(b)} This direct sum is uniquely determined by $(A,B)$, up to permutation of summands and replacement of $(A',B')$ by a similar pair $($i.e., by a pair obtained by similarity transformations$)$.
\end{lemma}

\begin{proof} (a)
There is a nonsingular $S$ such that
\begin{equation*}\label{ftyik}
S^{-1}(A,B)S=(A',B')\oplus(A'',B''),
\end{equation*}
where $A'$ is nilpotent and $A''$
is nonsingular. By \eqref{ffg1}, $B''=0$. Converting $A''$ to its Frobenius canonical form $\Phi_1\oplus\dots\oplus \Phi_r$, we obtain \eqref{teu}.

(b) Let
\[
R^{-1}((A',B')\oplus(A'',0))R =(C',D')\oplus(C'',0)
\]
where $C'$ is nilpotent and $C''$ is nonsingular. Then \[(A'\oplus A'')R=R(C'\oplus C'')\]
implies $R=R'\oplus R''$, and so \[(A',B')R'=R'(C',D'),\qquad
A''R''=R''C''.
\]
\vskip-1em
\end{proof}

Thus, we can suppose that $A$ is nilpotent. Then $0$ is the only eigenvalue of $A$, and so we can
reduce $A$ to its Jordan canonical form $J$  over any $\mathbb F$. Combine all Jordan
blocks of the same size into one
block, and obtain
\begin{equation}\label{kids}
J^{\text{\it +}}
:=J_{m_1}(0_{r_1})\oplus\dots \oplus
J_{m_t}(0_{r_t}),\qquad m_1< m_2<\dots<
m_t,
\end{equation}
in which
\begin{equation}\label{gte}
J_{m_i}(0_{r_i}):=\begin{bmatrix}
0_{r_i} &&&0\\I_{r_i}&0_{r_i}&&
\\&\ddots&\ddots&
\\ 0&&I_{r_i}&0_{r_i}
\end{bmatrix}  \qquad \text{($m_i^2$
blocks)}.
\end{equation}

Making
the same similarity transformations with $B$, we convert $(A,B)$ to some pair $(J^{\text{\it +}},C)$, which is similar to $(A,B)$. By \eqref{ffg1}, \begin{equation}\label{ffg8}
J^{\text{\it +}}C=CJ^{\text{\it +}}=0,
\end{equation}
hence, $C$ has the form
\begin{equation}\label{fry}
 C=\begin{bmatrix}C_{11}&\dots&C_{1t}\\
 \hdotsfor{3}\\
C_{t1}&\dots&C_{tt}\\
\end{bmatrix}
\end{equation}
(partitioned conformally to \eqref{kids}) in which
\begin{equation}\label{uyd}
C_{ij}=\begin{bmatrix}0&0&\dots&0\\
 \hdotsfor{4}\\
0&0&\dots&0\\
D_{ij}&0&\dots&0\\
\end{bmatrix}\qquad (m_im_j\ \text{blocks of size }r_i\times r_j);
\end{equation}
in particular, $C_{ii}$ is partitioned conformally to \eqref{gte}.
Combine all $D_{ij}$ into one matrix
\begin{equation}\label{jyt6}
D:=\begin{bmatrix}D_{11}&\dots&D_{1t}\\
 \hdotsfor{3}\\
D_{t1}&\dots&D_{tt}\\
\end{bmatrix}.
\end{equation}

We will reduce $(J^{\text{\it +}},C)$ by those similarity transformations $ S^{-1}(J^{\text{\it +}},C)S$ that preserve $J^{\text{\it +}}$; that is,
\begin{equation}\label{hytr}
C\mapsto C':=S^{-1}CS, \qquad S^{-1}J^{\text{\it +}}S=
J^{\text{\it +}}.
\end{equation}
Since $(J^{\text{\it +}},C)$ is similar to $(J^{\text{\it +}},C')$, \eqref{ffg8} implies $J^{\text{\it +}}C'=C'J^{\text{\it +}}=0$. Hence the matrix $C'$ has the form defined in \eqref{fry} and \eqref{uyd} with $C_{ij}$ and $D_{ij}$ replaced by $C'_{ij}$ and $D'_{ij}$.

Thus, transformations \eqref{hytr} preserve all (zero) blocks of $C$ outside of $D$.
In Lemma \ref{lem9n} we show that transformations \eqref{hytr} induce on $D$ the following matrix problem (each \emph{matrix problem} is given, by definition, by a set of matrices and a set of admissible transformations with these matrices; the question is to classify the equivalence classes of the set of matrices with respect to these admissible transformations; see \cite[Section 1.4]{g-r}).

 \begin{definition}\label{def}
The \emph{chessboard matrix problem} is given by
\begin{itemize}
  \item[\rm(a)]
the set of all block matrices  $D=[D_{ij}]$, in which some square blocks are scored along the main diagonal such that each horizontal or vertical strip contains at most one scored block, and

  \item[\rm(b)] the set of the following admissible transformations with each $D$:
\begin{itemize}
  \item[(i)] arbitrary elementary transformations within strips with the following restriction: each scored block is reduced by similarity transformations (i.e., we can make an elementary column transformation in any vertical strip, but if it contains a scored block then we must make the inverse row transformation in the horizontal strip containing this scored block);

\item[(ii)]
if $u$ is a column in vertical strip $i$, $v$ is a column in vertical strip $j$, and $i<j$, then we can replace $v$ by $v+\alpha u$, $\alpha\in\mathbb F$;

\item[(iii)]
if $u$ is a row in horizontal strip $i$, $v$ is a row in horizontal strip $j$, and $i<j$, then we can replace $v$ by $v+\alpha u$, $\alpha\in\mathbb F$.
\end{itemize}
Thus, all admissible additions between different strips are from left to right and from top to bottom.
\end{itemize}
\end{definition}

A canonical form with respect to transformations (i)--(iii) was obtained in \cite{n-r-s-b}. In particular, it was shown that each $D$ is reduced by transformations (i)--(iii) to a matrix with additional partition into blocks such that every horizontal or vertical strip contains at most one nonzero block and this block is nonsingular. We recall this reduction in Section \ref{sub2}; it will be used in Section
\ref{sub3}.

\begin{lemma}\label{lem9n}
Let $(A,B)$ be a pair of mutually annihilating matrices in which $A$ is nilpotent. Then $(A,B)$ is similar to some pair $(J^{\text{\it +}},C)$ in which $J^{\text{\it +}}$ is of the form \eqref{kids}.  If we restrict ourselves to those similarity transformations with $C$ that preserve $J^{\text{\it +}}$, then
\begin{itemize}
  \item
we obtain the chessboard matrix problem for the submatrix \eqref{jyt6} whose scored blocks are $D_{11},D_{22},\dots,D_{tt}$;
  \item
the blocks of $C$ outside of $D$ remain zero under these transformations.

\end{itemize}

\end{lemma}

\begin{proof}
We reduce $C$ by transformations \eqref{hytr}. Partition $S$ conformally to the partition of $J^{\text{\it +}}$ in \eqref{kids}:
\begin{equation}\label{rok}
S=\begin{bmatrix}S_{11}&\dots&S_{1t}\\
 \hdotsfor{3}\\
S_{t1}&\dots&S_{tt}\\
\end{bmatrix}.
\end{equation}

Since $S$ commutes with $J^{\text{\it +}}$, each $S_{ij}$ has the following form described in \cite[Chapter VIII, \S\,2]{gan}:
\begin{equation}\label{gyi}
\begin{bmatrix}
 R_{ij}&&&&&&&0\\
R'_{ij}&R_{ij}\\
 R''_{ij}&R'_{ij}&R_{ij}&
 \\
 \ddots&\ddots&\ddots&\ddots&
 \\
 \ddots&\ddots&R''_{ij}&R'_{ij}&R_{ij}
 \\
\end{bmatrix}\text{ or } \begin{bmatrix}
&&&&0\\
&\\
 R_{ij}\\
R'_{ij}&R_{ij}\\
 R''_{ij}&R'_{ij}&R_{ij}&
 \\
 \ddots&\ddots&\ddots&\ddots&
 \\
 \ddots&\ddots&R''_{ij}&R'_{ij}&R_{ij}
 \\
\end{bmatrix}
\end{equation}
(every $S_{ij}$ consists of $m_im_j$ blocks of size $r_i\times r_j$).

Substituting \eqref{fry} and \eqref{rok} into $SC'=CS$ and omitting zero entries, we obtain
\[
\begin{bmatrix}R_{11}&&0\\
\vdots&\ddots\\R_{t1}&\dots&R_{tt}
\end{bmatrix}\begin{bmatrix}D'_{11}&\dots&D'_{1t}\\
\hdotsfor{3}\\D'_{t1}&\dots&D'_{tt}
\end{bmatrix}=
\begin{bmatrix}D_{11}&\dots&D_{1t}\\
\hdotsfor{3}\\D_{t1}&\dots&D_{tt}
\end{bmatrix}
\begin{bmatrix}R_{11}&\dots&R_{1t}\\
&\ddots&\vdots\\0&&R_{tt}
\end{bmatrix}
\]
in which the first and the forth matrices are the submatrices of $S$ formed by the blocks of $S_{ij}$ at the positions $(m_i,m_j)$ and  $(1,1)$, respectively. Then
\begin{equation*}\label{gto}
\begin{bmatrix}D'_{11}&\dots&D'_{1t}\\
\hdotsfor{3}\\D'_{t1}&\dots&D'_{tt}
\end{bmatrix}=\begin{bmatrix}R_{11}^{-1}&&0\\
&\ddots&\\ *&&R_{tt}^{-1}
\end{bmatrix}
\begin{bmatrix}D_{11}&\dots&D_{1t}\\
\hdotsfor{3}\\D_{t1}&\dots&D_{tt}
\end{bmatrix}\begin{bmatrix}R_{11}&&*\\
&\ddots&\\ 0&&R_{tt}
\end{bmatrix}
\end{equation*}
where the stars denote arbitrary blocks. Thus, $D$ is reduced by transformations (i)--(iii) from Definition \ref{def}.
\end{proof}

\subsection*{Appendix: A proof of Lemma \ref{lem9n} by elementary transformations}\label{subsub1}

In this appendix we show that Lemma \ref{lem9n} can also be proved by elementary transformations. This primitive proof makes the reduction to chessboard matrix problem clearer, but the reader may omit it.

For simplicity, we assume that all Jordan blocks of $A$ are at most 3-by-3, the general case is considered analogously. Then
\begin{equation*}\label{jut3}
J= \underbrace{J_1(0)\oplus
\dots \oplus J_1(0)}_{\mbox{$p$
times}}\oplus \underbrace{J_2(0)\oplus
\dots \oplus J_2(0)}_{\mbox{$q$
times}}\oplus \underbrace{J_3(0)\oplus
\dots \oplus J_3(0)}_{\mbox{$r$
times}},
\end{equation*}
where $p,q,r$ are natural numbers or zero. The pair $(A,B)$ is similar to $(J^{\text{\it
+}},C)$, in which
\begin{equation}\label{iod}
J^{\text{\it
+}}= J_1(0_p) \oplus
J_2(0_q) \oplus
J_3(0_r)=
\left[\begin{array}{c|cc|ccc}
 0_p& &&&&0\\ \hline
& 0_q&0_q&&&\\
&I_q& 0_q&&&\\  \hline
&&& 0_r&0_r&0_r\\
&&& I_r& 0_r&0_r\\
0&&&0_r&I_r& 0_r
\end{array}\right].
\end{equation}

Since $J^{\text{\it
+}}C=CJ^{\text{\it
+}}=0$,
\begin{equation*}\label{hte}
C=\left[\begin{array}{c|cc|ccc}
 D_{11}&D_{12}&0&D_{13}&0&0\\ \hline
 0&0&0&0&0&0\\
  D_{21}&D_{22}&0&D_{23}&0&0\\ \hline
 0&0&0&0&0&0\\
 0&0&0&0&0&0\\
 D_{31}&D_{32}&0&D_{33}&0&0\\
\end{array}\right].
\end{equation*}

Let us prove that the similarity transformations with $(J^{\text{\it
+}},C)$
that preserve $J^{\text{\it
+}}$ induce on
 \begin{equation*}\label{jyt}
D:=\begin{bmatrix}\lefteqn{\diagdown}D_{11}&D_{12}&D_{13}\\
D_{21}&\lefteqn{\diagdown}D_{22} &D_{23}\\D_{31}&D_{32}& \lefteqn{\diagdown}D_{33}
\end{bmatrix}
\end{equation*}
(in which the blocks $D_{11},D_{22},D_{33}$ are scored)
the chessboard matrix problem.

\begin{itemize}
  \item[(i)] We can make with $D$ transformations (i) from Definition \ref{def} using the following transformations within six horizontal and six vertical strips of $J^{\text{\it
+}}$ and $C$:

\begin{itemize}
  \item
Any elementary column transformation in the first vertical strip of $J^{\text{\it
+}}$ and $C$ simultaneously, and then the inverse row transformation.

  \item
Any elementary column transformation in the second vertical strip of $J^{\text{\it
+}}$ and $C$ and then the inverse row transformation. The latter transformation spoils the identity block at the position (3,2) in \eqref{iod}, we restore it by the inverse row transformation in the third horizontal strip and then make the initial column transformation in the third vertical strip.

  \item
Any elementary column transformation in vertical strips 4, 5, 6 simultaneously, then the inverse row transformation in horizontal strips 4, 5, 6.
\end{itemize}
Thus, $D_{11},\ D_{22}$, and $D_{33}$ are reduced by similarity transformations.

  \item[(ii)]  We can make with $D$ transformations (ii) from Definition \ref{def} as follows.
We can add a column of vertical strip 1 in $D$ to a column of vertical strip 2 or 3 since the corresponding transformation in $J^{\text{\it
+}}$ and the inverse row transformation do not change $J^{\text{\it
+}}$. We can add a column of vertical strip 2 in $D$ to a column of vertical strip 3; the corresponding transformation in $J^{\text{\it
+}}$ may spoil the zero block $(3,4)$,  we restore it by adding rows of horizontal strip 5 and the inverse column transformations do not change $J^{\text{\it
+}}$.

  \item[(iii)]  We can make with $D$ transformations (iii) from Definition \ref{def} as follows. We can add a row of horizontal strip 1 in $D$ to a row of horizontal strip  2 or 3 since the corresponding transformation in $J^{\text{\it
+}}$ does not change $J^{\text{\it
+}}$. We can add a row of horizontal strip 2 in $D$ to a row of horizontal strip 3; the corresponding transformation in $J^{\text{\it
+}}$ may spoil the zero block $(6,2)$, we restore it by adding columns of vertical strip 5.

\end{itemize}

\section{Solving the chessboard matrix problem}\label{sub2}

In this section we prove the following lemma

\begin{lemma}\label{lemu}
Let $D$ be a block matrix in which some square blocks are scored along the main diagonal such that each horizontal or vertical strip contains at most one scored block. Then there is an algorithm that
\begin{itemize}
  \item[\rm(a)]
using transformations {\rm(i)--(iii)} from Definition \ref{def} and
  \item[\rm(b)]
making additional partition of strips into substrips such that the partition of each scored block  into horizontal substrips duplicates its partition into vertical substrips $($i.e., all diagonal subblocks of scored blocks are square$)$
\end{itemize}
transforms $D$ into a matrix $D_0$ partitioned into subblocks such that
\begin{equation}                 \label{lye}
\parbox{25em}
{each horizontal or vertical substrip contains at most one nonzero subblock and this subblock is  nonsingular.}
\end{equation}
\end{lemma}

\begin{proof}
We use induction on the size of $D$.
If the first horizontal strip of $D$ is zero then we can delete it reducing the size of $D$.
Hence we can suppose that the first horizontal strip of $D=[D_{ij}]$ is nonzero; let $D_{1k}$ be the first nonzero block.
\medskip

\noindent \emph{Case 1: $D_{1k}$ is not scored}. By transformations (i) from Definition \ref{def} we reduce it to the form
\begin{equation}\label{hyt}
\tilde D_{1k}:=
\left[\begin{array}{c|c}
 I_r&0\\ \hline
0&0
\end{array}\right]
\end{equation}
By adding linear combinations of columns of $D$ that cross $I_r$ (transformations (ii)), we make zero all entries to the right of $I_r$.
By adding linear combinations of rows that cross $I_r$ (transformations (iii)), we make zero all entries under $I_r$. Extending the partition \eqref{hyt}, we divide the first horizontal strip into two substrips and the $k$th vertical strip into two substrips. If the new horizontal or vertical partition goes through the scored block, then we make the perpendicular partition (vertical or horizontal, respectively) such that this block is partitioned into 4 subblocks with square diagonal subblocks scored along the main diagonal. Denote the obtained matrix by $\tilde D$. For example, if the new horizontal partition and the new vertical partition go through scored blocks $F$ and $G$, then
\begin{equation*}\label{fri}
\tilde D=\begin{MAT}(c){3ccc3c0c3ccc3c0c3ccc3}
\first3
 0&\dots&0 & I_r&0 & 0&\dots&0 & \lefteqn{\!\diagdown}0&0 & 0&\dots&0
        \\0
 0&\dots&0 & 0&0 & *&\dots&* & F_{21}&\lefteqn{\,\diagdown}F_{22} & *&\dots&*
        \\3
 && & \begin{matrix}
 0\\[-3mm] \vdots\\[-1mm] 0
 \end{matrix}&\begin{matrix}
 *\\[-3mm] \vdots\\[-1mm] *
 \end{matrix} &&&&&&&&
        \\3
 && & \lefteqn{\!\diagdown}0&G_{12} &&&&&&&&
        \\0
 && & 0& \lefteqn{\,\diagdown}G_{22} &&&&&&&&
        \\3
&& & \begin{matrix}
 0\\[-3mm] \vdots\\[-1mm] 0
 \end{matrix}&\begin{matrix}
 *\\[-3mm] \vdots\\[-1mm] *
 \end{matrix} &&&&&&&&
        \\3
\end{MAT}
\end{equation*}

By the \emph{canonical substrips} we mean the horizontal substrip containing $I_r$ and the vertical substrip containing $I_r$ (we call them ``canonical'' since they are substrips of the canonical form of $D$ with respect to transformations (i)--(iii)). Denote by $E$ the block matrix obtained from $\tilde D$ by deleting the
canonical substrips.

We will reduce $\tilde D$ by those transformations (i)--(iii) (defined with respect to the initial partition of $D$ into blocks) that preserve $\tilde D_{1k}$ and the canonical substrips. Let us show that these transformations induce on $E$ the chessboard matrix problem.
\begin{itemize}
  \item
Using transformations (i) we can add columns of the first vertical substrip that goes through $F$ to columns of the second vertical substrip that goes through $F$. Since $F$ is scored, we must make the inverse row transformation in $F$. This may spoil the zero subblocks of $F$ above $F_{21}$ and $F_{22}$, we restore them by adding linear combinations of columns of $I_r$.

  \item
We can add rows of the first horizontal substrip that goes through $G$ to rows of the second horizontal substrip. The inverse transformation of columns in $G$ may spoil the zero subblocks of  $G$ to the left of $G_{12}$ and $G_{22}$, we restore them by adding linear combinations of rows of $I_r$.
\end{itemize}
By induction on the size, Lemma \ref{lemu} holds for $E$; that is, $E$ is reduced to a block matrix $E_0$ satisfying the condition \eqref{lye}.
Replacing $E$ by $E_0$ in $\tilde D$ and making the additional partitions into subblocks in accordance with the additional partitions in $E_0$, we obtain a block matrix $D_0$ satisfying \eqref{lye}.
\medskip

\noindent \emph{Case 2: $D_{1k}$ is scored and nonnilpotent}. We may reduce it by similarity transformations.

Convert $D_{1k}$ to the form $K\oplus N$ in which $K$ is a nonsingular Frobenius matrix and $N$ is nilpotent (if $D_{1k}$ is nonsingular then $N$ does not appear). Using transformations (ii) and (iii), we make zero all entries to the right of $K$ and under $K$ and obtain the matrix
\begin{equation}\label{fre}
\tilde D:=\begin{MAT}(c){3ccc3c0c3ccc3}
\first3
 0&\dots&0 &\lefteqn{\diagdown}K&0 & 0&\dots&0
        \\0
 0&\dots&0 & 0&\lefteqn{\diagdown}N & *&\dots&*
        \\3
 &*& & \begin{matrix}
 0\\[-3mm] \vdots\\[-1mm] 0
 \end{matrix}&\begin{matrix}
 *\\[-3mm] \vdots\\[-1mm] *
 \end{matrix} &&*&
        \\3
\end{MAT}
\end{equation}
Denote by $E$ the block matrix obtained from \eqref{fre} by deleting
the horizontal and vertical substrips containing $K$.

We will reduce  \eqref{fre} by those transformations (i)--(iii) that preserve the zeros to the right of $K$ and under $K$ and that transform $K$ into a nonsingular matrix and $N$ into a nilpotent matrix. These transformations induce on $E$ the chessboard matrix problem.
By induction on the size, Lemma \ref{lemu} holds for $E$; that is, $E$ is converted to a block matrix $E_0$ satisfying the condition \eqref{lye}.
Replacing $E$ by $E_0$ in \eqref{fre}, we obtain a block matrix $D_0$  satisfying \eqref{lye}.
\medskip

\noindent \emph{Case 3: $D_{1k}$ is scored and nilpotent}.
Reduce it by similarity transformations to the form
\begin{equation}\label{kidu}
\tilde D_{1k}:=
J_{m_1}(0_{r_1})\oplus\dots \oplus
J_{m_t}(0_{r_t}),\qquad m_1> m_2>\dots>
m_t,
\end{equation}
in which $J_{m_i}(0_{r_i})$ is defined in \eqref{gte}. Using transformations (ii) and (iii), make zero all entries to the right of $I_{r_i}$ and under $I_{r_i}$ for each $I_{r_i}$ in \eqref{gte}; which
converts $D$ to the form
\begin{equation}\label{kue}
\tilde D:=\begin{MAT}(@){3c3c3c3}
\first3 &\phantom{a}&\\
 0&\ \ \tilde D_{1k}\ \  &{\cal F}\\ &\phantom{a}&\\3
   \; *\; &{\cal G}&\; *\;
 \\3
\end{MAT}=\begin{MAT}(@){3c3ccc3c3}
\first3
 0&J_{m_1}(0_{r_1})& &0 &{\cal F}_1
 \\
\aligntop
 \vdots&&\ddots& &\vdots\\ &&&&\\
 0&0& &J_{m_t}(0_{r_t}) &{\cal F}_t\\3
 &&&&\\
  \ \ *\ \ &{\cal G}_1&\dots &{\cal G}_t&\ \ *\ \
 \\ &&&&\\3
\end{MAT}
\end{equation}
(we do not draw partitions into strips except for the strips of $\tilde D_{1k}$), in which
\[
{\cal F}_i=\begin{bmatrix}F_i\\ 0\\ \vdots \\0
\end{bmatrix}\ \text{ ($m_i$ strips)},\qquad {\cal G}_i=\begin{bmatrix}0&\dots& 0& G_i
\end{bmatrix}\ \text{ ($m_i$ strips)}
\]
for $i=1,\dots, t$.

By the \emph{canonical substrips} we mean all horizontal and vertical substrips of $\tilde D$ that contain $I_{r_i}$ from $\tilde D_{1k}$. Delete in $\tilde D$ the
canonical substrips and obtain the block matrix
\begin{equation}\label{kui}
E=\begin{MAT}(@){3c3c0c0c3c3}
\first3
 0&\lefteqn{\diagdown}0_{r_1}& \dots&0 &{F}_1\\0\aligntop
 \vdots& \vdots&\ddots&  \vdots&\vdots\\0
 0&0&  \dots&\lefteqn{\diagdown}0_{r_t} &{F}_t\\3
 &&&&\\
  \ \ *\ \ &{G}_1&\dots &{G}_t&\ \ *\ \
 \\ &&&&\\3
\end{MAT}
\end{equation}
partitioned as $D$,
in  which we divide additionally the first horizontal strip and the $k$th vertical strip into $t$ substrips of sizes $r_1,\dots, r_t$; the line  in $\tilde D$  that scores $\tilde D_{1k}$ along its diagonal becomes the line  in $E$ that scores the diagonal blocks $0_{r_1},\dots, 0_{r_t}$ along their diagonals.

Let us show that
\begin{equation}                 \label{lut}
\parbox{25em}
{the transformations (i)--(iii) from Definition \ref{def} with  $\tilde D$ that preserve the canonical substrips and $\tilde D_{1k}$ induce the chessboard matrix problem on $E$.}
\end{equation}

The first horizontal strip of $\tilde D$ is reduced by transformations that preserve $\tilde D_{1k}$:
\begin{equation*}\label{rek}
(S^{-1}\oplus I) \tilde D (I\oplus S\oplus I)=\tilde D', \qquad S^{-1}\tilde D_{1k}S=\tilde D_{1k}.
\end{equation*}
By the latter equality and \eqref{kidu}, $S$ has the form defined in \eqref{rok} and \eqref{gyi}. Then
\[
(S\oplus I)\tilde D'
=\tilde D (I\oplus S\oplus I)
\]
implies
\[
\Big(\begin{bmatrix}R_{11}&&0\\
\vdots&\ddots\\R_{t1}&\dots&R_{tt}
\end{bmatrix}\oplus I\Big)E'=
E \Big(I\oplus \begin{bmatrix}R_{11}&\dots&R_{1t}\\
&\ddots&\vdots\\0&&R_{tt}
\end{bmatrix}\oplus I\Big)
\]
(in which $E'$ is defined by \eqref{kui} with $F_i$ and $G_i$ replaced by $F_i'$ and $G_i'$),
and so
\begin{equation*}\label{hyo}
E'=
\Big(\begin{bmatrix}R_{11}^{-1}&&0\\
&\ddots\\ *&&R_{tt}^{-1}
\end{bmatrix}\oplus I\Big)
E \Big(I\oplus \begin{bmatrix}R_{11}&&*\\
&\ddots&\\0&&R_{tt}
\end{bmatrix}\oplus I\Big),
\end{equation*}
where the stars denote arbitrary blocks. Thus, the substrips of the first horizontal and the $k$th vertical strips of $E$ are reduced by transformations (i)--(iii) from Definition \ref{def}, which proves \eqref{lut}.

By induction on the size, $E$ converts to a block matrix $E_0$ satisfying the condition \eqref{lye}.
Replacing $E$ by $E_0$ in \eqref{kue}, we obtain a block matrix satisfying \eqref{lye}.
\end{proof}

\subsection*{Appendix: A proof of \eqref{lut} by elementary transformations}\label{subsub2}

The key statement in the proof of Lemma \ref{lemu} is \eqref{lut}. In this appendix we derive the statement \eqref{lut} using elementary transformations (the reader may omit it).

For simplicity, we use
\begin{equation*}\label{kia}
\tilde D_{1k}:=
J_{5}(0_{p})\oplus
J_{3}(0_{q})
\end{equation*}
in place of \eqref{kidu}, then the matrix \eqref{kue} takes the form
\begin{equation*}
\tilde D=\begin{MAT}(@){3c3c3c3}
\first3 &\phantom{a}&\\
 0&\ \ \tilde D_{1k}\ \  &{\cal F}\\ &\phantom{a}&\\3
   \; *\; &{\cal G}&\; *\;
 \\3
\end{MAT}=
\begin{MAT}(@){rcccccccccc}
&& \scriptstyle  1& \scriptstyle 2& \scriptstyle 3& \scriptstyle 4   & \scriptstyle 5& \scriptstyle 6 & \scriptstyle 7 & \scriptstyle 8&\\
\scriptstyle  1&0 & 0_p&&&   && & && F_1\\ \scriptstyle  2
&0 & I_p&0_p&&&& & && 0\\ \scriptstyle  3
&0 & &I_p&0_p&&& & 0&& 0\\ \scriptstyle  4
&0 & &&I_p&0_p&& & && 0\\ \scriptstyle  5
&0 & &&&I_p&0_p& & && 0\\ \scriptstyle  6
&0 & & &&& & 0_q&& & F_2\\ \scriptstyle  7
&0 & & &0&& & I_q&0_q&&  0\\
\scriptstyle  8
&0 & & &&& & &I_q&0_q& 0\\ &&& &&& &&& &\\
&\ \ *\ \  &\; 0\; &\; 0\; &\; 0\; &\; 0\; &G_1& \;0\; &\;0\;&G_2&\ \  *\ \ \\ & && &&& &&&&
\addpath{(1,0,3)rrrrrrrrrruuuuuuuuuuullllllllllddddddddddd}
\addpath{(1,3,3)rrrrrrrrrr}
\addpath{(1,6,1)rrrrrrrrrr}
\addpath{(2,0,3)uuuuuuuuuuu}
\addpath{(10,0,3)uuuuuuuuuuu}
\addpath{(7,0,1)uuuuuuuuuuu}
\\
\end{MAT}\ .
\end{equation*}

Let us restrict ourselves to those transformations (i)--(iii) from Definition \ref{def} with $\tilde D$ that preserve the canonical substrips and $\tilde D_{1k}$, and prove that they induce the chessboard matrix problem on the submatrix\\[-7mm]
\[
\begin{matrix}
\\
E=
\end{matrix}
\begin{MAT}(e){rcccc}
\alignbottom
&&\scriptstyle  1& \scriptstyle 6 &\\
\scriptstyle  1 &0&\lefteqn{\diagdown}0_p&0&F_1\\
\scriptstyle  6 &0&0&\lefteqn{\diagdown}0_q&F_2\\
&*&G_1&G_2&*
\addpath{(1,0,3)rrrruuullllddd}
\addpath{(1,2,0)rrrr}
\addpath{(1,1,3)rrrr}
\addpath{(2,0,3)uuu}
\addpath{(3,0,0)uuu}
\addpath{(4,0,3)uuu}
\\
\end{MAT}
\]
in which the blocks $0_p$ and $0_q$ are scored.

(i) We can make with $E$ transformations (i) from Definition \ref{def}  using the following sequence of transformations in $\tilde D$:
\begin{itemize}
  \item[--]
First, make any elementary row transformation in $F_1$. Since the block $\tilde D_{1k}$ is scored, we must make the inverse column transformation in vertical substrip 1 of $\tilde D_{1k}$. This spoils the subblock $I_p$ in position (2,1) of $\tilde D_{1k}$, we restore it by the initial row transformation in horizontal substrip 2. The inverse column transformation spoils $I_p$ in position (3,2), we restore it by the initial row transformation in horizontal substrip 3, and so on. Thus, preserving the submatrix $J_5(0_p)$ in $\tilde D_{1k}$, we must make any elementary transformation of rows in horizontal substrips 1, 2, 3, 4, 5 simultaneously (in particular, of rows of $F_1$) and then the inverse transformation of columns in vertical substrips 1, 2, 3, 4, 5 (in particular, of columns of $G_1$), and so the subblock $0_p$ in $E$ is scored.

  \item[--]
Analogously, we can make any elementary row transformation in horizontal substrips 6, 7, 8 and then the inverse column transformation in vertical substrips 6, 7, 8, and so the subblock $0_q$ in $E$ is scored.
\end{itemize}

(ii) For each $p\times q$ matrix $S$, we can replace $G_2$ by $G_2+G_1S$ as follows. Add vertical substrips 3, 4, and 5 of $\tilde D_{1k}$, multiplied on the right by $S$, to vertical substrips 6, 7, and 8, respectively:
\begin{equation*}\label{fit}
\begin{MAT}(@){rrcccccccccc}
&& & &&&&\lefteqn{\mspace{-27mu}\xrightarrow{\ \cir (S\oplus S\oplus S)\ }}& &&&
  \\
&& &&&\lefteqn{\mspace{-9mu}\overbrace{\qquad\qquad}} &&&\lefteqn{\mspace{-9mu}\overbrace{\qquad\qquad}}  &&&
   \\
&&& \scriptstyle  1& \scriptstyle 2& \scriptstyle 3& \scriptstyle 4   & \scriptstyle 5& \scriptstyle 6 & \scriptstyle 7 & \scriptstyle 8&
    \\
&\scriptstyle  1&0 & 0_p&&&   && & && F_1
     \\
&\scriptstyle  2&0 & I_p&0_p&&&& & && 0
    \\
&\scriptstyle  3&0 & &I_p&0_p&&& & && 0
    \\
\righteqn&\scriptstyle  4&0 & &&I_p&0_p&& & && 0
    \\
 &\scriptstyle  5&0 & &&&I_p&0_p& & && 0
    \\
{\begin{picture}(9,6)
   \put(-55,7){$\scriptstyle -(S\oplus S\oplus S)\cir$}
   \put(-5,0){\vector(0,1){20}}
\end{picture}}
&\scriptstyle  6&0 & &&&&& 0_q&&& F_2
    \\
\righteqn&\scriptstyle  7&0 & &&&&& I_q&0_q&&  0
    \\
&\scriptstyle  8&0 & &&&&& &I_q&0_q& 0\\
 &&&&&&&&&&&
     \\
&&\ \ *\ \  &\; 0\;&\; 0\; &\; 0\; &\; 0\; &G_1& \;0\; &\;0\;&G_2&\ \  *\ \
   \\ &&&&&&&&&&&
\addpath{(2,0,3)rrrrrrrrrruuuuuuuuuuullllllllllddddddddddd}
\addpath{(2,3,3)rrrrrrrrrr}
\addpath{(2,6,2)rrrrrrrrrr}
\addpath{(2,9,0)rrrrrrrrrr}
\addpath{(3,0,3)uuuuuuuuuuu}
\addpath{(5,0,0)uuuuuuuuuuu}
\addpath{(11,0,3)uuuuuuuuuuu}
\addpath{(8,0,2)uuuuuuuuuuu}
\\
\end{MAT}
\end{equation*}
This transformation replaces $G_2$ by $G_2+G_1S$ but also replaces the zero subblocks (4,6) and (5,7) of $\tilde D_{1k}$ by $S$. The inverse row transformation restores the zero subblocks (4,6) and (5,7) but spoils zero subblocks of
horizontal substrip 3 to the right of $\tilde D_{1k}$. We restore them by adding linear combinations of columns of $I_p$. Therefore, $E$ can be reduced by transformations (ii) from Definition \ref{def}.

(iii) For each $q\times p$ matrix $S$, we can replace $F_2$ by $F_2+SF_1$ as follows. Add horizontal substrips 1, 2, and 3, multiplied on the left by $S$, to horizontal substrips 6, 7, and 8, respectively:
\begin{equation*}\label{fie}
\begin{MAT}(@){rrcccccccccc}
& && &&&\lefteqn{\mspace{-33mu}\xleftarrow{\ \cir (-S\oplus S\oplus S)\ }}& &&&&
  \\
&& &\lefteqn{\mspace{-9mu}\overbrace{\qquad\qquad}} &&&&&\lefteqn{\mspace{-9mu}\overbrace{\qquad\qquad}}  &&&
   \\
&&& \scriptstyle  1& \scriptstyle 2& \scriptstyle 3& \scriptstyle 4   & \scriptstyle 5& \scriptstyle 6 & \scriptstyle 7 & \scriptstyle 8&
  \\
&\scriptstyle  1&0 & 0_p&&&   && & && F_1\\
\righteqn&\scriptstyle 2&0 & I_p&0_p&&&& & && 0\\
&\scriptstyle 3&0 & &I_p&0_p&&& & && 0\\
&\scriptstyle 4&0 & &&I_p&0_p&& & && 0\\
{\begin{picture}(9,6)
   \put(-50,7){$\scriptstyle (S\oplus S\oplus S)\cir$}
   \put(-5,38){\vector(0,-1){55}}
\end{picture}}
 &\scriptstyle 5&0 & &&&I_p&0_p& & && 0\\
&\scriptstyle 6&0 & &&&&& 0_q&&& F_2
\\  \righteqn&\scriptstyle  7&0 & &&&&& I_q&0_q&&  0\\
&\scriptstyle 8&0 & &&&&& &I_q&0_q& 0\\
&&&&&&&&&&&\\
&&\ \ *\ \  &\; 0\;&\; 0\; &\; 0\; &\; 0\; &G_1& \;0\; &\;0\;&G_2&\ \  *\ \ \\ &&&&&&&&&&&
\addpath{(2,0,3)rrrrrrrrrruuuuuuuuuuullllllllllddddddddddd}
\addpath{(2,3,3)rrrrrrrrrr}
\addpath{(2,6,2)rrrrrrrrrr}
\addpath{(2,8,0)rrrrrrrrrr}
\addpath{(3,0,3)uuuuuuuuuuu}
\addpath{(6,0,0)uuuuuuuuuuu}
\addpath{(11,0,3)uuuuuuuuuuu}
\addpath{(8,0,2)uuuuuuuuuuu}
\\
\end{MAT}
\end{equation*}
This transformation replaces $F_2$ by $F_2+SF_1$  but also replaces the zero  subblocks (7,1) and (8,2) in $\tilde D_{1k}$ by $S$. The inverse column transformation restores subblocks (7,1) and (8,2) but spoils
zero subblocks in
vertical substrip 3 below $\tilde D_{1k}$. We restore them by adding linear combinations of rows of $I_p$. Therefore, $E$ can be reduced by transformations (iii) from Definition \ref{def}.

\section{Proof of Theorem \ref{theor}}\label{sub3}

Let  ${\cal A}$ and ${\cal B}$ be mutually annihilating operators on a vector space $V$ over a field $\mathbb F$. Let $A$ and
$B$ be their matrices in some basis of
$V$. Changing the basis, we can reduce $(A,B)$ by {similarity transformations} \eqref{fek1}. By Lemma \ref{lem9}(a), $(A,B)$  is similar to \eqref{teu}, in which every summand $(\Phi_i,0_{n_i})$ is of cycle type: it is given by the ordinary loop $\circlearrowleft$ associated with $\Phi_i$.
The direct sum \eqref{teu} is uniquely determined by $(A,B)$, up to permutation of summands and replacement of $(A',B')$ by a similar pair.

Hence, it suffices to prove Theorem \ref{theor} for pairs $({\cal A},{\cal B})$ in which ${\cal A}$ is nilpotent.
Then $A$ is nilpotent too.

\begin{lemma}\label{lej}
Let $(A,B)$ be a pair of mutually annihilating matrices in which $A$ is nilpotent.
Then the algorithm from Sections \ref{sub1} and \ref{sub2} reduces $(A,B)$ by similarity transformations to a pair $(A_0,B_0)$ of matrices that can be conformally partitioned into blocks such that
each horizontal or vertical strip contains at most one nonzero block and this block is nonsingular.
\end{lemma}

\begin{proof}  By Lemmas \ref{lem9n}  and \ref{lemu}, $(A,B)$ is similar to some pair $(A_0,B_0):=(J^{\text{\it +}},C)$  of block matrices in which $J^{\text{\it +}}$ is of the form \eqref{kids} and the submatrix $D$ of $C$ defined in \eqref{uyd} and \eqref{jyt6} is additionally partitioned into subblocks such that each horizontal or vertical substrip  contains at most one nonzero subblock and this subblock is nonsingular.

Since all diagonal blocks $D_{11},D_{22},\dots,D_{tt}$ of $D$ are scored, by Lemma \ref{lemu}(b) the diagonal subblocks of each $D_{ii}$ (with respect to the new partition) are square; that is, the partition of $D_{ii}$ into horizontal substrips coincides with its partition into vertical substrips. Each $C_{ii}$ in \eqref{uyd} consists of $m_i^2$ square blocks of the same size and one of them is $D_{ii}$; we partition each block of $C_{ii}$ into subblocks conformally to the partition of $D_{ii}$ and extend this partition to the whole $C$. Since all subblocks of $C$ outside of $D$ are zero, each horizontal or vertical substrip of $C$  contains at most one nonzero subblock and this subblock is nonsingular.

Partition $J^{\text{\it +}}$ into subblocks conformally to the partition of $C$ into subblocks.
The partition \eqref{gte} of each $J_{m_i}(0_{r_i})$ into blocks is conformal to the partition \eqref{uyd} of $C_{ii}$ into blocks; moreover, the partition of each $I_{r_i}$ in $J_{m_i}(0_{r_i})$ is conformal to the partition of $D_{ii}$ into subblocks. Thus, all diagonal subblocks of $I_{r_i}$ are square; i.e., they are the identity matrices.
\end{proof}

Let $({\cal A},{\cal B})$ be given by a pair $(A_0,B_0)$ of block matrices  described in Lemma \ref{lej}.
Decompose the vector space  $V$ into the direct sum
\begin{equation}\label{kud}
V=V_1\oplus\dots\oplus V_t
\end{equation}
conformally to the partition
of $A_0$ and $B_0$ into blocks.

Let us construct a graph $\Gamma$  with vertices
$1,2,\dots,t$, ordinary arrows
$\longrightarrow$, and double arrows $\Longrightarrow$, as follows.
If block $(i,j)$ of $A_0$ is nonzero,
then ${\cal A}V_j\subset V_i$, we draw
$j\longrightarrow i$. If block $(i,j)$ of $B_0$ is nonzero, then ${\cal
B}V_j\subset V_i$, we draw
$j\Longrightarrow i$.

Thus, the number of
arrows is equal to the number of
nonzero blocks in $A_0$ and $B_0$.
The number of arrows in each
vertex $j$ is at most $2$. If it is
$2$ then there are only 3 possibilities for the behaviour of arrows in $j$:
\begin{equation*}\label{y4j}
i\longrightarrow j\longrightarrow k,\qquad
i\longrightarrow j\Longleftarrow k,\qquad i\Longleftarrow
j\Longleftarrow k,
\end{equation*}
because
\begin{itemize}
  \item
the cases $i\longrightarrow j\longleftarrow k$ and
$i\Longrightarrow  j\Longleftarrow k$ are
impossible since each horizontal strip
contains at most one nonzero block,
  \item
the cases $i\longleftarrow j\longrightarrow k$ and $
i\Longleftarrow j\Longrightarrow k$ are impossible since each vertical strip
contains at most one nonzero block,

  \item
the cases $i\longrightarrow j\Longrightarrow k$ and
$i\Longrightarrow  j\longrightarrow k$ are impossible
by  ${\cal A}{\cal B}={\cal B}{\cal A}=0$.

  \end{itemize}
Therefore, each connected component of the graph $\Gamma$  is either a path graph \eqref{jut} or a cycle graph \eqref{jut1} (up to renumeration of vertices).

Let $\Gamma_1,\dots,\Gamma_r$ be all connected components of $\Gamma$.
For each $\Gamma_l$, denote by $W_l$ the direct sum of all spaces $V_i$ from \eqref{kud} that correspond to the vertices of $\Gamma_l$. Clearly, $W_l$ is invariant under the operators $\cal A$ and $\cal B$. Denote by  ${\cal A}_l$ and ${\cal B}_l$ their restrictions on $W_l$. Then
\begin{align}\label{tew}
V=W_1\oplus\dots\oplus W_r,\qquad
({\cal A},{\cal B})= ({\cal A}_1,{\cal B}_1)\oplus\dots\oplus ({\cal A}_r,{\cal B}_r).
\end{align}

\emph{Case 1: $r=1$}. Then $\Gamma$ is of the form \eqref{jut} or \eqref{jut1}, each vertex $i$ is assigned by the vector space $V_i$ and each arrow $i\text{ --- } [i]$ with  \[
[1]:=2,\ \dots,\ [t-1]:=t,\quad [t]:=1
\]
is associated with the linear bijection ${\cal F}_i$ between the corresponding vector spaces, which is induced by $\cal A$ or $\cal B$. Starting from a basis in $V_1$ and taking the images or preimages with respect to ${\cal F}_1,\dots,{\cal F}_{t-1}$, we sequentially construct the bases in $V_2,\dots,V_t$. The linear bijections ${\cal F}_1,\dots,{\cal F}_{t-1}$ are given in these bases by the identity matrices
\begin{equation}\label{yre}
F_1=\dots=F_{t-1}=I_d,\qquad d:=\dim V_1.
\end{equation}

If $\Gamma$ is a path graph, then the pair $({\cal A},{\cal B})$ is the direct sum of $d$ pairs of path type (see Definition \ref{kie}).

Let $\Gamma$ be a cycle graph of the form \eqref{jut1}.

If \eqref{jut1} is periodic (see Definition
\ref{rsi}(ii)), then we make it aperiodic as follows. The sequence $(c_1,\dots,c_t)$ defined by \eqref{ytj} is periodic; i.e.,
 \[
 (c_1,\dots,c_t)= (c_1,\dots,c_{\tau}; c_{\tau+1},\dots,c_{2\tau};\dots ;c_{(q-1)\tau+1}, \dots,c_{q\tau}).
 \]
for some $\tau<t$ that divides $t$.
Let $\tau $ be the minimal number with this property. Replace $\Gamma$ by the graph that is defined by the sequence
 $(c_1,\dots,c_{\tau})$ and replace each $V_i$ ($i=1,\dots,\tau$) by $V_i\oplus V_{i+\tau}\oplus V_{i+2\tau}\oplus\dots$.
 The obtained graph is aperiodic and gives the same pair of mutually annihilating operators.

Thus, $\Gamma$ is aperiodic.
Choose other bases in $V_1,\dots,V_t$ using transition matrices $S_1,\dots,S_t$. Then $F_i$ changes by the rule
\begin{equation*}
F'_i=
  \begin{cases}
S^{-1}_{[i]}F_iS_i & \text{if $i \longrightarrow [i]$}\,, \\
S^{-1}_{i}F_iS_{[i]} & \text{if $i \Longleftarrow [i]$}\,,
  \end{cases}\qquad i=1,\dots,t,
\end{equation*}
and so the matrix
\begin{equation*}
G_i:=
  \begin{cases}
F_i^{-1} & \text{if $i \longrightarrow [i]$} \\
F_i & \text{if $i \Longleftarrow [i]$}
  \end{cases}
\end{equation*}
changes by the rule
\begin{equation*}\label{dgte}
G'_i=
S^{-1}_{i}G_iS_{[i]},\qquad i=1,\dots,t.
\end{equation*}

If $S_1=\dots=S_t$, then the matrices \eqref{yre} do not change and $G_t$ is reduced by similarity transformations. Convert it to the Frobenius canonical matrix
\begin{equation}\label{loy}
\Phi=\Phi_1\oplus\dots\oplus \Phi_p
\end{equation}
in which every $\Phi_i$ is an $n_i\times n_i$ Frobenius block of the form \eqref{3},
and obtain
\begin{equation}\label{gree}
(G'_1,\dots,G'_t)=(I,\dots,I,\Phi).
\end{equation}
Then taking
\[
(S_1,\dots,S_t)=(I,\dots,I,\Phi,\dots,\Phi)
\]
we may convert \eqref{gree} into
\begin{equation}\label{greed}
(G''_1,\dots,G''_t)=(I,\dots,I,\Phi, I,\dots,I)
 \end{equation}
with $\Phi$ at any position.

Since $\Gamma $ is aperiodic, it contains at least one double arrow; otherwise it is the ordinary loop $\circlearrowleft$ associated with a nonsingular matrix, but this is impossible since $\cal A$ is nilpotent.
Let $i \Longleftarrow [i]$ be any double arrow. By \eqref{greed} with $\Phi$ at the position $i$, we can associate $G''_i=F''_i=\Phi$ with this arrow. By \eqref{loy}, $({\cal A},{\cal B})$ is the direct sum of $p$ pairs of cycle type (see Definition \ref{rsi}).
\medskip

\emph{Case 2: $r>1$}.
Each pair
$({\cal A}_l,{\cal B}_l)$ in the decomposition \eqref{tew} corresponds to the connected graph $\Gamma _l$. Reasoning as in Case 1, we decompose
$({\cal A}_l,{\cal B}_l)$ into a direct sum of pairs of path or cyclic type.

Thus, we have decomposed $({\cal A},{\cal B})$ into a direct sum of
pairs of path and cycle types, which proves the existence of the decomposition from Theorem \ref{theor}.
By \eqref{greed}, for each cyclic graph that corresponds to a summand of cyclic type, we can transfer the Frobenius block associated with a double arrow to any other double arrow.

This decomposition is uniquely determined by $({\cal A},{\cal B})$ up to transformations  (i) and (ii) from Theorem \ref{theor}(a) since
each direct summand is indecomposable and distinct summands are isomorphic if and only if they are of cyclic type and the corresponding cyclic graphs coincide up to transformations (ii). Hence, we can use the Krull--Schmidt theorem \cite[Chapter 1, Theorem 3.6]{bas}, which ensures that each quiver representation is isomorphic to a direct sum of indecomposable representations determined uniquely up to isomorphism of summands. Hence, each system of linear mappings uniquely decomposes into a direct sum of indecomposable systems, up to isomorphism of summands (moreover, by \cite[Theorem 2]{ser_izv} each system of bilinear forms and linear mappings over $\mathbb C$ and $\mathbb R$ uniquely decomposes into a direct sum of indecomposable systems, up to isomorphism of summands).

This proves the statement (a) of Theorem \ref{theor}.
The statement (b) follows from (a) since two pairs of linear operators are isomorphic if and only if their matrix pairs are similar.

\section*{Acknowledgment}

The authors would like to thank Alexander E. Guterman for a very careful reading of the manuscript and many helpful suggestions.

\end{document}